\documentclass[a4paper,10pt]{amsart}
\usepackage[utf8]{inputenc}

\usepackage{amsmath}
\usepackage{amsthm}
\usepackage{amssymb}

\usepackage{mathrsfs}
\usepackage[all]{xy}

\newcommand{\PP}{\mathcal{P}}

\newcommand{\set}[2]{\left\{ #1 \mid #2 \right\}}

\newcommand{\TT}{T}
\newcommand{\pp}{\overline{p}}
\newcommand{\nn}{m}
\newcommand{\sgn}{\operatorname{sgn}}

\newcommand{\zetaL}{T}
\newcommand{\zetaLsym}{T^\Sigma}
\newcommand{\zetaA}{S}
\newcommand{\zetaAsym}{S^\Sigma}

\newcommand{\LL}{ \mathsf L}
\newcommand{\Ahat}{\hat{\mathsf{A}}}
\newcommand{\pt}{ \mathsf p}
\newcommand{\hh}{h}

\title{Hirzebruch $L$-polynomials and multiple zeta values}

\author{Alexander Berglund}
\author{Jonas Bergstr\"om}

\address{Department of Mathematics\\
Stockholm University\\
SE-106 91 Stockholm\\
Sweden}
\email{alexb@math.su.se}
\email{jonasb@math.su.se}

\newtheorem{theorem}{Theorem}

\newtheorem{corollary}[theorem]{Corollary}
\newtheorem{lemma}[theorem]{Lemma}

\theoremstyle{definition}

\begin{document}
\begin{abstract}
We express the coefficients of the Hirzebruch $L$-polynomials in terms of certain alternating multiple zeta values. In particular, we show that every monomial in the Pontryagin classes appears with a non-zero coefficient, with the expected sign. Similar results hold for the polynomials associated to the $\hat{A}$-genus.
\end{abstract}

\maketitle

\section{Introduction}
The Hirzebruch $L$-polynomials are certain polynomials with rational coefficients,
\begin{align*}
\LL_1 & = \frac{1}{3} \pt_1, \\
\LL_2 & = \frac{1}{45} \big( 7\pt_2 -\pt_1^2 \big), \\
\LL_3 & = \frac{1}{945} \big( 62 \pt_3 -13\pt_2\pt_1 +2\pt_1^3 \big), \\
\vdots
\end{align*}
featured in the Hirzebruch signature theorem, which expresses the signature $\sigma(M)$ of a smooth compact oriented manifold $M^{4k}$ as
$$\sigma(M) = \langle \LL_k, [M] \rangle,$$
where $\pt_i$ are taken to be the Pontryagin classes of the tangent bundle of $M$, see \cite[Theorem 8.2.2]{Hirzebruch} or \cite[Theorem 19.4]{MS}. The $k$th polynomial has the form
$$\LL_k = \LL_k(\pt_1,\ldots,\pt_k) = \sum \hh_{j_1,\ldots,j_r} \pt_{j_1}\cdots \pt_{j_r},$$
where the sum is over all partitions $(j_1,\ldots,j_r)$ of $k$, i.e., sequences of integers $j_1\geq \cdots \geq j_r  \geq 1$
such that $j_1+\cdots + j_r =k$.
The purpose of this note is to establish certain properties of the coefficients $\hh_{j_1,\ldots,j_r}$.

For real numbers $s_1,\ldots,s_r>1$, we define the series
$$\zetaL(s_1,\ldots,s_r) = \sum_{n_1\geq_2 \cdots \geq_2 n_r\geq 1} \frac{(-1)^{n_1+\cdots + n_r}}{n_1^{s_1}\cdots n_r^{s_r}},$$
where $n\geq_2 m$ means ``$n\geq m$ with equality only if $n$ is even''. Define the symmetrization of this series by
$$\zetaLsym(s_1,\ldots, s_r) = \sum_{\sigma\in \Sigma_r} \zetaL(s_{\sigma_1},\ldots,s_{\sigma_r}),$$
where $\Sigma_r$ is the symmetric group.

\begin{theorem} \label{thm:main}
The coefficients of the Hirzebruch $L$-polynomials are given by
$$\hh_{j_1,\ldots,j_r} = \frac{(-1)^r}{\alpha_1!\cdots \alpha_k!}  \frac{2^{2k}}{\pi^{2k}} \zetaLsym(2j_1,\ldots,2 j_r),$$
where $\alpha_\ell$ counts how many of $j_1,\ldots,j_r$ are equal to $\ell$.
\end{theorem}

It is well-known that $\hh_k$ is positive for all $k$. In \cite[Appendix A]{Weiss}, it is argued that $\hh_{i,j}$ is always negative and that $\hh_{i,j,k}$ is always positive (following an argument attributed to Galatius in the case of $\hh_{i,j}$), and it is asked whether it has been proved in general that $(-1)^{r-1}\hh_{j_1,\ldots,j_r}$ is positive. We have not been able to locate such a result in the literature, but we can prove it using our formula.
It follows from the following result.

\begin{theorem} \label{thm:positivity}
For all real $s_1,\ldots,s_r>1$,
$$\zetaL(s_1,\ldots,s_r) <  0.$$
\end{theorem}

\begin{corollary} \label{cor:main}
The coefficient $\hh_{j_1,\ldots,j_r}$ in the Hirzebruch $L$-polynomial $\LL_k$ is non-zero for every partition $(j_1,\ldots,j_r)$ of $k$. It is negative if $r$ is even and positive if $r$ is odd.
\end{corollary}

It is remarked in \cite{Weiss} that a similar pattern has been observed in the multiplicative sequence of polynomials associated with the $\hat{A}$-genus. The polynomials in question are
\begin{align*}
\Ahat_1 & = - \frac{1}{24} \pt_1, \\
\Ahat_2 & = \frac{1}{5760} \big( - 4\pt_2  + 7 \pt_1^2 \big), \\
\Ahat_3 & = \frac{1}{967680} \big( -16 \pt_3 +44\pt_2\pt_1 -31\pt_1^3 \big), \\
\vdots
\end{align*}
These can be treated similarly. Let us write
$$\Ahat_k = \sum a_{j_1,\ldots,j_r} \pt_{j_1} \cdots \pt_{j_r},$$
where the sum is over all partitions $(j_1,\ldots,j_r)$ of $k$. Consider the series
$$\zetaA(s_1,\ldots,s_r) = \sum_{n_1\geq \cdots \geq n_r \geq 1} \frac{1}{n_1^{s_1} \cdots n_r^{s_r}},$$
and its symmetrization
$$\zetaAsym(s_1,\ldots, s_r) = \sum_{\sigma\in\Sigma_r} \zetaA(s_{\sigma_1},\ldots,s_{\sigma_r}).$$
\begin{theorem} \label{thm:A-hat}
The coefficients of the $\hat{A}$-polynomials are given by
$$a_{j_1,\ldots,j_r} = \frac{(-1)^r}{\alpha_1! \cdots \alpha_k!} \frac{1}{(2\pi)^{2k}} \zetaAsym(2j_1,\ldots,2 j_r).$$
In particular, the coefficient $a_{j_1,\ldots,j_r}$ is negative if $r$ is odd and positive if $r$ is even.
\end{theorem}

\section{Proofs}
The first step in our proof is to establish a formula that expresses the coefficient $\hh_{j_1,\ldots,j_r}$ as a linear combination of products $\hh_{k_1}\cdots \hh_{k_\ell}$.
This generalizes the formulas for $\hh_{i,j}$ and $\hh_{i,j,k}$ found in \cite{Weiss}. In the appendix of \cite{FS}, recursive formulas for computing $\hh_{j_1,\ldots,j_r}$ in terms of products  $\hh_{k_1}\cdots \hh_{k_\ell}$ are given. Here we give an explicit closed formula. The result holds for arbitrary multiplicative sequences of polynomials (see \cite[\S1]{Hirzebruch}).

\begin{theorem} \label{thm:multiplicative genus}
Let $K_0,K_1,K_2,\ldots$ be a multiplicative sequence of polynomials with
$$K_k = \sum \lambda_{j_1,\ldots,j_r} \pt_{j_1}\cdots\pt_{j_r}.$$
The coefficients satisfy the relation
\begin{equation} \label{eq:hirzebruch}
\lambda_{j_1,\ldots,j_r} = \frac{1}{\alpha_1!\cdots \alpha_k!} \sum_{\mathcal{P}} (-1)^{r-\ell}c_\PP \, \lambda_{k_1}\cdots \lambda_{k_\ell},
\end{equation}
where $\alpha_i$ counts how many of $j_1,\ldots,j_r$ are equal to $i$, the sum is over all partitions $\mathcal{P} = \{P_1,\ldots,P_\ell\}$ of the set $\{1,2,\ldots,r\}$,
$$c_\PP = \big(|P_1|-1\big)!\cdots \big(|P_\ell|-1\big)!,$$
and
$$k_m = \sum_{i\in P_m} j_i.$$
\end{theorem}

\begin{proof}
A multiplicative sequence of polynomials is determined by its characteristic power series
$$Q(z) = \sum_{k=0}^\infty b_k z^k,$$
where $b_k = \lambda_{1,\ldots,1}$ is the coefficient of $\pt_1^k$ in $K_k$. Indeed, if we, as in \cite{Hirzebruch}, formally interpret the coefficients $b_k$ as elementary symmetric functions in $\beta_1',\ldots,\beta_m'$ ($m\geq k$), so that
$$1+b_1z +b_2z^2 +\cdots + b_mz^m = (1+\beta_1'z)\cdots(1+\beta_m'z),$$
then the coefficient $\lambda_{j_1,\ldots,j_r}$ is the \emph{monomial symmetric function} in $\beta_1',\ldots,\beta_m'$ (see~\cite[Lemma 1.4.1]{Hirzebruch}).

Note that $\lambda_k$ equals the power sum $\sum_i (\beta_i')^k$. The product $\lambda_{k_1}\cdots \lambda_{k_\ell}$ is then the \emph{power sum symmetric function} evaluated at $\beta_i'$, and the claim follows from a general formula that expresses the monomial symmetric functions in terms of power sum symmetric functions, see Theorem \ref{thm:mp} below.
\end{proof}

The characteristic series of the Hirzebruch $L$-polynomials is
$$\frac{\sqrt{z}}{\tanh \sqrt{z}} = 1 + \sum_{k=1}^\infty b_k z^k,$$
where 
$$b_k = (-1)^{k-1} \frac{2^{2k}}{(2k)!}B_k,$$
and $B_k$ are the Bernoulli numbers,
$$B_1 =\frac{1}{6},\,\, B_2 = \frac{1}{30},\,\, B_3 = \frac{1}{42}, \ldots,$$
see \cite[\S1.5]{Hirzebruch}.

As is well-known, the leading coefficient $\hh_k$ of $\pt_k$ in $\LL_k$ is given by
\begin{equation} \label{eq:leading coefficient}
\hh_k = \frac{2^{2k}\big(2^{2k-1}-1\big)}{(2k)!}B_k,
\end{equation}
see \cite[p.12]{Hirzebruch}. In \cite{Weiss}, the formula
$$\hh_k = \zeta(2k)\frac{2^{2k}-2}{\pi^{2k}},$$
involving the Riemann zeta function,
$$
\zeta(s) = \sum_{n=1}^\infty \frac{1}{n^s},
$$
is used to argue that $h_{i,j}<0$ and $h_{i,j,k}>0$. From this point on, our argument will depart from that of \cite{Weiss}. A key observation is that we can express $\hh_k$ in terms of the alternating zeta function,
$$\zeta^*(s)  = \sum_{n=1}^\infty \frac{(-1)^{n-1}}{n^s},$$
instead of the Riemann zeta function. It is well-known, and easily seen, that 
$$\zeta^*(s) = \big(1-2^{1-s}\big)\zeta(s).$$
Moreover, the following holds for all positive integers $k$,
\begin{equation} \label{eq:eta-bernoulli}
\zeta^*(2k) = \frac{\pi^{2k}(2^{2k-1}-1)}{(2k)!}B_k.
\end{equation}
Combining this with \eqref{eq:leading coefficient}, we see that
\begin{equation}
\hh_k = \frac{2^{2k}}{\pi^{2k}} \zeta^*(2k).
\end{equation}
From \eqref{eq:hirzebruch} we get
\begin{equation} \label{eq:hof}
\hh_{j_1,\ldots,j_r} = \frac{1}{\alpha_1!\cdots \alpha_k!} \frac{2^{2k}}{\pi^{2k}}\sum_{\mathcal{P}} (-1)^{r-\ell}c_\PP \, \zeta^*(2k_1)\cdots \zeta^*(2k_\ell),
\end{equation}
where the notation is as in Theorem \ref{thm:multiplicative genus}.

The next observation is that the sum in the right hand side bears a striking resemblance with the right hand side of Hoffman's formula \cite[Theorem 2.2]{Hoffman} (proved anew in Theorem \ref{thm:hoffman} below), which relates multiple zeta values and products of zeta values --- the only difference is that $\zeta^*$ appears instead of $\zeta$. 
The second step in the proof is then to find a Hoffman-like formula for $\zeta^*$. Here is what we were led to write down: Define, for real numbers $s_1,\ldots,s_r>1$,
$$\zetaL(s_1,\ldots,s_r) = \sum_{n_1\geq_2 \cdots \geq_2 n_r\geq 1} \frac{(-1)^{n_1+\cdots + n_r }}{n_1^{s_1}\cdots n_r^{s_r}},$$ 
where $n\geq_2 m$ means ``$n\geq m$ with equality only if $n$ is even''. Then symmetrize, and define
$$\zetaLsym(s_1, \ldots,s_r) = \sum_{\sigma\in \Sigma_r} \zetaL(s_{\sigma_1},\ldots,s_{\sigma_r}).$$
Here is our Hoffman-like formula. Together with \eqref{eq:hof} it implies Theorem \ref{thm:main}.

\begin{theorem} \label{thm:multiple eta}
The following equality holds for all real $s_1,\ldots,s_r>1$, 
\begin{equation} \label{eq:multiple eta}
\sum_{\PP} (-1)^{r-\ell} c_\PP \, \zeta^*(\underline{s},\PP) = (-1)^{r}\zetaLsym(s_1, \ldots,s_r),
\end{equation}
where the sum is over all partitions $\PP = \{P_1,\ldots,P_\ell\}$ of $\{1,2,\ldots,r\}$ and
$$c_{\PP} = (|P_1|-1)!\cdots (|P_\ell|-1)!, \quad \zeta^*(\underline{s},\PP)=\zeta^*\big(\sum_{i\in P_1} s_i\big) \cdots \zeta^*\big(\sum_{i\in P_\ell} s_i\big).$$

\end{theorem}

\begin{proof}
This will follow by specialization of Theorem \ref{thm:Tp} below.
\end{proof}

Next we turn to the proof of Theorem \ref{thm:positivity}, which says that
$$\zetaL(s_1,\ldots,s_r) < 0$$
for all real $s_1,\ldots,s_r>1$.

\begin{proof}[Proof of Theorem \ref{thm:positivity}]
The proof is in principle not more difficult than the proof that $\zeta^*(s)$ is positive; to see this one simply arranges the sum as
$$\zeta^*(s) = \left(1-\frac{1}{2^s}\right) + \left(\frac{1}{3^s} - \frac{1}{4^s}\right) +\cdots + \left(\frac{1}{(2k-1)^s} - \frac{1}{(2k)^s}\right) + \cdots,$$
and notes that the summands are positive. Since the series is absolutely convergent, we are free to rearrange as we please, as the reader will recall from elementary analysis.

Towards the general case, introduce for $k\geq 1$ the auxiliary series
$$
\zetaL_{2k}(s_1,\ldots,s_r) = \sum_{n_1\geq_2 \cdots \geq_2 n_r \geq 2k} \frac{(-1)^{n_1 +\cdots + n_r}}{n_1^{s_1} \cdots n_r^{s_r}}.
$$
Then one can argue using the following two equalities, whose verification we leave to the reader:
\begin{equation*} \label{eq:eta1}
\zetaL(s_1,\ldots,s_r) = \sum_{k\geq 1} \left( - \frac{1}{\big(2k-1\big)^{s_r}} + \frac{1}{\big(2k\big)^{s_r}}\right) \zetaL_{2k}(s_1,\ldots,s_{r-1}),
\end{equation*}
\begin{align*} \label{eq:eta2}
  \zetaL_{2k}(s_1&,\ldots, s_r) = \\ & \sum_{\ell \geq k} \sum_{j=1}^r
  \frac{1}{(2\ell)^{s_r} \cdots (2\ell)^{s_{j+1}}} \left( \frac{1}{\big(2\ell\big)^{s_j}} - \frac{1}{\big(2\ell+1\big)^{s_j}}\right) \zetaL_{2\ell+2}(s_1,\ldots,s_{j-1}).
\end{align*}
Here $T_{2\ell+2}(s_1,\ldots,s_{j-1})$ should be interpreted as $1$ for $j=1$.
The second equality may be used to show that $\zetaL_{2k}(s_1,\ldots,s_r)$ is positive by induction on $r$. The first equality then shows that $\zetaL(s_1,\ldots,s_r)$ is negative.
\end{proof}

Finally, we turn to the proof of Theorem \ref{thm:A-hat}. The argument turns out to be easier in this case. Recall that the $\hat{A}$-genus has characteristic series
$$Q(z) = \frac{\sqrt{z}/2}{\sinh (\sqrt{z}/2)}.$$
Let us write
$$\Ahat_k = \sum a_{j_1,\ldots,j_r} \pt_{j_1} \cdots \pt_{j_r},$$
where the sum is over all partitions $(j_1,\ldots,j_r)$ of $k$.
By using the Cauchy formula (see \cite[p.11]{Hirzebruch})
one can calculate the coefficient $a_k$ of $\pt_k$ in $\Ahat_k$. The result is
$$a_k = \frac{(-1)^k}{2(2k)!} B_k = - \frac{1}{(2\pi)^{2k}}\zeta(2k).$$
It follows that
$$a_{k_1} \cdots a_{k_\ell} = \frac{(-1)^{\ell}}{(2\pi)^{2k}} \zeta(2k_1) \cdots \zeta(2k_\ell),$$
for every partition $(k_1,\ldots,k_\ell)$ of $k$. Theorem \ref{thm:multiplicative genus} then yields
\begin{align*}
a_{j_1,\ldots, j_r} & = \frac{1}{\alpha_1!\cdots \alpha_k!}\sum_\PP (-1)^{r-\ell} c_\PP a_{k_1}\cdots a_{k_\ell} \\
& = \frac{1}{\alpha_1!\cdots \alpha_k!} \frac{(-1)^r}{(2\pi)^{2k}} \sum_\PP c_\PP \zeta(2k_1) \cdots \zeta(2k_\ell).
\end{align*}
The terms in the sum are clearly positive, so we see already from this expression that $(-1)^ra_{j_1,\ldots,j_r}>0$ for all partitions $(j_1,\ldots,j_r)$ of $k$. However, more can be said; the sum in the right hand side now not only resembles but is \emph{equal} to the right hand side of another formula of Hoffman \cite[Theorem 2.1]{Hoffman}. In our notation this formula says that
$$
\zetaAsym(s_1,\ldots,s_r) = \sum_\PP c_\PP \zeta(\underline{s},\PP).
$$
This proves Theorem \ref{thm:A-hat}.

\section{Combinatorics of infinite sums}
The proofs of Theorem \ref{thm:multiplicative genus} and Theorem  \ref{thm:multiple eta}, as well as of Hoffman's formula, share the same combinatorial underpinnings; this is the topic of the present section.

Recall that a \emph{partition} of a set $S$ is a set of non-empty disjoint subsets,
$$\pi = \{\pi_1,\ldots,\pi_r\},$$
such that $S = \pi_1\cup \cdots \cup \pi_r$. Write $\ell(\pi) = r$ for the \emph{length} of $\pi$. The set of partitions $\Pi_S$ is partially ordered by refinement,
$\pi = \{\pi_1\ldots,\pi_r\} \leq \rho = \{\rho_1,\ldots,\rho_\ell\}$ if and only if there is a partition $\PP = \{P_1,\ldots,P_\ell\}$ of the set $\{1,2,\ldots,r\}$ such that
\begin{equation} \label{eq:order relation}
\rho_i = \bigcup_{j\in P_i} \pi_j,\quad 1\leq i\leq \ell.
\end{equation}
We will write $\rho = \PP(\pi)$ if \eqref{eq:order relation} holds. Note that for every $\rho\geq \pi$ there is a unique partition $\PP$ such that $\rho = \PP(\pi)$.

We will consider certain formal power series in indeterminates $a_n$ for $a\in S$ and positive integers $n$. For a subset $T\subseteq S$, write
$$f_T(n) = \prod_{a\in T} a_n.$$
For a partition $\pi =  \{\pi_1,\ldots,\pi_r\}$ of $S$, consider the formal power series
\begin{align*}
p_\pi & = \sum_{n_1,\ldots,n_r} f_{\pi_1}(n_1) \cdots f_{\pi_r}(n_r), \\
m_\pi & = \sum_{\stackrel{n_1,\ldots,n_r}{distinct}} f_{\pi_1}(n_1) \cdots f_{\pi_r}(n_r). \\
\end{align*}
It is then immediate that
$$p_\pi = \sum_{\rho \geq \pi} m_\rho.$$
By applying the M\"obius inversion formula (see e.g.~\cite[Proposition 3.7.2]{Stanley1}), we get
\begin{equation} \label{eq:mp}
m_\pi = \sum_{\rho \geq \pi} \mu(\pi,\rho) p_\rho.
\end{equation}
The M\"obius function of $\Pi_S$ is given by
$$\mu(\pi,\rho) = (-1)^{\ell(\pi)-\ell(\rho)}(b_1-1)!\ldots (b_{\ell(\rho)}-1)!,$$
where the number
$$b_i = b_i(\pi,\rho) = |P_i|$$
counts how many `$\pi$-blocks' $\rho_i$ consists of, see e.g.~\cite[Example 3.10.4]{Stanley1}.

Let us first note that this gives a neat proof of Hoffman's formula (though we would be surprised if this has not been noticed before). Recall that the multiple zeta function is defined by
$$\zeta(s_1,\ldots,s_r) = \sum_{n_1>\cdots > n_r\geq 1} \frac{1}{n_1^{s_1} \cdots n_r^{s_r}},$$
for real $s_1,\ldots,s_r >1$.
\begin{theorem}[Hoffman {\cite[Theorem 2.2]{Hoffman}}] \label{thm:hoffman}
$$\sum_{\sigma \in \Sigma_r} \zeta(s_{\sigma_1},\ldots,s_{\sigma_r}) = \sum_{\PP} (-1)^{r-\ell} c_{\PP} \, \zeta(\underline{s},\PP),$$
where the sum is over all partitions $\PP = \{P_1,\ldots,P_\ell\}$ of $\{1,2,\ldots,r\}$ and
$$c_{\PP} = (|P_1|-1)!\cdots (|P_\ell|-1)!, \quad \zeta(\underline{s},\PP)=\zeta\big(\sum_{i\in P_1} s_i\big) \cdots \zeta\big(\sum_{i\in P_\ell} s_i\big).$$
\end{theorem}

\begin{proof}
Take $S= \{1,2,\ldots,r\}$ and substitute $a_n$ by $\frac{1}{n^{s_a}}$ for $a\in S$ in \eqref{eq:mp}.
\end{proof}

Secondly, we will use \eqref{eq:mp} to express the monomial symmetric functions in terms of power sum symmetric functions. We refer to \cite[Chapter 7]{Stanley2} for a pleasant introduction to symmetric functions. For an integer partition $I = (i_1,\ldots,i_r)\vdash k$, recall that the \emph{power sum symmetric function} $p_I$ is the formal power series in indeterminates $x_1,x_2,\ldots$ defined by $p_I = p_{i_1}\cdots p_{i_r}$, where
$$p_j = \sum_i x_i^j.$$
The  \emph{monomial symmetric function} $m_I$ is defined as the sum of all pairwise distinct monomials of the form $x_{\sigma_1}^{i_1}\cdots x_{\sigma_r}^{i_r}$.

\begin{theorem} \label{thm:mp}
For every $k\geq 1$ and every integer partition $I = (i_1,\ldots,i_r)$ of $k$,
$$m_I = \frac{1}{\alpha_1!\cdots \alpha_k!} \sum_{\mathcal{P}} (-1)^{r-\ell} c_{\mathcal{P}}  \, p_{J},$$
where the sum is over all partitions $\PP = \{P_1,\ldots,P_\ell\}$ of the set $\{1,2,\ldots,r\}$, the number $\alpha_j$ counts how many of $i_1,\ldots,i_r$ are equal to $j$, and $J = (j_1,\ldots,j_\ell)$ is given by
$$j_u = \sum_{v\in P_u} i_v,\quad 1\leq u\leq \ell.$$
\end{theorem}

\begin{proof}
Let $S$ be any set with $k$ elements. Perform the substitution $a_n = x_n$ for each $a\in S$ in the equality \eqref{eq:mp}
and note that this takes $p_\pi$ to $p_I$ and $m_\pi$ to $\alpha_1!\cdots \alpha_k!\, m_I$, where $I = \big(|\pi_1|,\ldots,|\pi_r|\big)$ is the integer partition underlying the set partition $\pi$ (assuming, as we may, $|\pi_1|\geq \cdots \geq |\pi_r|$).
\end{proof}

Next, we turn to the result that will specialize to our formula in Theorem \ref{thm:multiple eta}. Consider the following alternating version of $p_\pi$:
$$\pp_\pi = \sum_{n_1,\ldots,n_r} (-1)^{n_1+\cdots + n_r} f_{\pi_1}(n_1)\cdots f_{\pi_r}(n_r).$$
For an \emph{ordered} partition $\widetilde{\pi} = (\pi_1,\ldots,\pi_r)$ we define
$$\TT_{\widetilde{\pi}} = \sum_{n_1\geq_2 \cdots \geq_2 n_r} (-1)^{n_1+\cdots + n_r} f_{\pi_1}(n_1)\cdots f_{\pi_r}(n_r).$$
Then for an unordered partition $\pi = \{\pi_1, \ldots, \pi_r\}$, define
$$\TT_\pi^\Sigma = \sum_{\widetilde{\pi}} \TT_{\widetilde{\pi}},$$
where the sum is over the $r!$ ordered partitions $\widetilde{\pi}$ whose underlying unordered partition is $\pi$.

\begin{lemma} \label{lemma:length sum}
For every partition $\pi\in \Pi_S$,
$$\sum_{\rho\geq \pi} (-1)^{\ell(\rho)} \ell(\rho)! = (-1)^{\ell(\pi)}.$$
\end{lemma}

\begin{proof}
We may without loss of generality assume that $\pi$ is the minimal element, because the poset $\set{\rho\in \Pi_S}{\rho \geq \pi}$ is isomorphic to the poset $\Pi_\pi$ of partitions of the set $\pi$. Then $n= \ell(\pi)$ is the number of elements of $S$. The number of partitions of length $k$ in $\Pi_S$ is equal to the Stirling number of the second kind $S(n,k)$, see e.g.~\cite[Example 3.10.4]{Stanley1}. Thus,
\begin{equation} \label{eq:stirling}
\sum_{\rho\in \Pi_S} (-1)^{\ell(\rho)} \ell(\rho)! = \sum_{k=1}^n (-1)^k S(n,k) k!.
\end{equation}
By plugging in $x=-1$ in the well-known identity
$$\sum_{k =1}^n S(n,k) (x)_k = x^n,$$
where $(x)_k = x(x-1)(x-2) \cdots (x-k+1)$, we see that \eqref{eq:stirling} equals $(-1)^n$.
\end{proof}

\begin{theorem} \label{thm:Tp}
For every partition $\pi$,
\begin{equation} \label{eq:Tp}
(-1)^{\ell(\pi)}\TT_\pi^\Sigma = \sum_{\rho \geq \pi} (-1)^{\ell(\rho)}\mu(\pi,\rho) \pp_\rho.
\end{equation}
\end{theorem}

\begin{proof}
By M\"obius inversion, the equality is equivalent to
\begin{equation} \label{eq:pT}
(-1)^{\ell(\pi)} \pp_\pi = \sum_{\rho \geq \pi} (-1)^{\ell(\rho)} \TT_\rho^\Sigma,
\end{equation}
and we proceed to prove \eqref{eq:pT}. It is clear that both sides can be written as linear combinations of series of the form
$$\nn_{\nu,e} = \sum_{\stackrel{n_1,\ldots,n_m\,\, distinct}{n_i \equiv_2 e_i}} f_{\nu_1}(n_1)\cdots f_{\nu_m}(n_m),$$
for various $\nu = \{\nu_1,\ldots,\nu_m\} \geq \pi$, where $e$ is an assignment of a parity $e_i \in\{0,1\}$ to each $\nu_i$. For example, if $\nu= \{\{a,b\},\{c\}\}$ and $e$ assigns $1$ to $\{a,b\}$ and $0$ to $\{c\}$, then
$$\nn_{\nu,e} = \sum_{n_1 \, odd,\,\, n_2 \, even} a_{n_1}b_{n_1} c_{n_2}.$$
The question is with what coefficients $\nn_{\nu,e}$ will appear in the respective sides of \eqref{eq:pT}.
For the left hand side this is not difficult: $\nn_{\nu,e}$ appears in $(-1)^{\ell(\pi)} \pp_\pi$ with coefficient
\begin{equation} \label{eq:pn}
\sgn(\pi,\nu,e) = (-1)^{v_1 (e_1-1) + \cdots + v_m (e_m-1)},
\end{equation}
where $v_i$ is the number of $\pi$-blocks in $\nu_i$.

The right hand side requires a little more effort --- and notation. It is clear that $\nn_{\nu,e}$ appears in $\TT_\rho^\Sigma$ only if $\nu \geq \rho$ and $e$ assigns an even value to $\nu_i$ whenever $\nu_i$ consists of more than one $\rho$-block. This can be reformulated as saying that $\nu\geq \rho \geq e(\nu)$, where $e(\nu)\leq \nu$ is the partition that keeps $\nu_i$ intact if $e_i$ is odd and splits $\nu_i$ completely if $e_i$ is even. Or more precisely, $e(\nu)$ is the smallest element below $\nu$ that contains $\nu_i$ whenever $e_i$ is odd.
Since we symmetrize, there will be repetitions; for $\nu\geq \rho\geq e(\nu)$, the term involving $\nn_{\nu,e}$ will be repeated $b_{\rho,\nu} = b_1! \cdots b_m!$ times in $\TT_\rho^\Sigma$, where $b_i$ is the number of $\rho$-blocks in $\nu_i$. Thus, the coefficient of $\nn_{\nu,e}$ in $(-1)^{\ell(\rho)} \TT_\rho^\Sigma$ is $\sgn(\rho,\nu,e)b_{\rho,\nu}$. It follows that the coefficient of $\nn_{\nu,e}$ in $\sum_{\rho\geq \pi}  (-1)^{\ell(\rho)} \TT_\rho^\Sigma$ is
\begin{equation} \label{eq:Tn}
\sum_{\nu \geq \rho \geq e(\nu) \vee \pi} \sgn(\rho,\nu,e)b_{\rho,\nu} = \sum_{\nu \geq \rho \geq e(\nu) \vee \pi}  (-1)^{b_1(e_1-1) + \cdots + b_m(e_m-1)} b_1!\cdots b_m!,
\end{equation}
where $e(\nu)\vee \pi$ is the least upper bound of $e(\nu)$ and $\pi$. 

Put $\pi_{(i)}=\{\pi_j \in \pi : \pi_j \subseteq \nu_i\}$ and $\nu_{(i)} = \{\nu_i\}$. We then have an isomorphism of posets $[e(\nu) \vee \pi,\nu] \cong \prod_{i : e_i = 0} [\pi_{(i)},\nu_{(i)}]$. Under this isomorphism $\rho \in [e(\nu) \vee \pi,\nu] $ is sent to $\rho_{(i)}=\{\rho_j \in \rho : \rho_j \subseteq \nu_i\} \in [\pi_{(i)},\nu_{(i)}]$. Note that $b_i$ is the length of $\rho_{(i)}$. We now find that the sum \eqref{eq:Tn} factors as a product
$$\prod_{i : e_i = 0} \sum_{\nu_{(i)}\geq \rho_{(i)}\geq \pi_{(i)}} (-1)^{b_i}b_i!.$$
By Lemma \ref{lemma:length sum}, this is equal to
$$\prod_{i : e_i = 0}  (-1)^{v_i}.$$
This shows that  \eqref{eq:Tn} equals \eqref{eq:pn}, and the theorem is proved.
\end{proof}

To prove Theorem \ref{thm:multiple eta}, take $S=\{1,2,\ldots,r\}$ and substitute $a_n$ by $\frac{1}{n^{s_a}}$ in \eqref{eq:Tp}.

\subsection*{Acknowledgments}
We thank Don Zagier and Matthias Kreck for valuable comments.
The impetus for this work was a question from Oscar Randal-Williams to the first author about certain points in \cite{BM}. The first author was supported by the Swedish Research Council through grant no.~2015-03991.

\end{document}